\documentclass[12pt,reqno]{amsart}

\usepackage{times}
\usepackage[T1]{fontenc}
\usepackage{mathrsfs}
\usepackage{latexsym}
\usepackage[dvips]{graphics}
\usepackage[dvips]{graphicx}
\usepackage{epsfig}
\usepackage{hyperref, amsmath,amsfonts,amsthm,amssymb,amscd, cleveref}
\usepackage{color}

\newcommand{\white}[1]{ {\color[rgb]{1,1,1} #1 } }

\newtheorem{thm}{Theorem}[section]
\newtheorem{conj}[thm]{Conjecture}

\newtheorem{lem}[thm]{Lemma}
\newtheorem{prop}[thm]{Proposition}

\newtheorem{defn}[thm]{Definition}

\numberwithin{equation}{section}
\newcommand{\tensor}{\otimes}

\newcommand{\con}{\equiv}
\newcommand{\isom}{\cong}

\newcommand{\surj}{\twoheadrightarrow}
\newcommand{\inj}{\hookrightarrow}
\newcommand{\s}{\sigma}

\newcommand{\e}{\varepsilon}

\newcommand{\g}{\gamma}
\newcommand{\R}{\mathbb{R}}
\newcommand{\Z}{\mathbb{Z}}
\newcommand{\Q}{\mathbb{Q}}
\newcommand{\C}{\mathbb{C}}

\DeclareMathOperator{\Heis}{H}
\renewcommand{\H}{\Heis}

\renewcommand{\aa}{\mathfrak{a}}

\newcommand{\cc}{\mathfrak{c}}

\newcommand{\OO}{\mathcal{O}}

\DeclareMathOperator{\Gal}{Gal}

\DeclareMathOperator{\GL}{GL}

\DeclareMathOperator{\U}{U}

\DeclareMathOperator{\M}{M}

\DeclareMathOperator{\Tr}{Tr}

\DeclareMathOperator{\Real}{Re}

\renewcommand{\Re}{\Real}

\newcommand{\ds}{\displaystyle}

\newcommand{\ol}{\overline}

\newcommand{\comment}[1]{}
\newcommand{\bi}{\begin{itemize}}
\newcommand{\ei}{\end{itemize}}
\newcommand{\ben}{\begin{enumerate}}
\newcommand{\een}{\end{enumerate}}
\newcommand{\be}{\begin{equation}}
\newcommand{\ee}{\end{equation}}
\newcommand{\bea}{\begin{eqnarray}}
\newcommand{\eea}{\end{eqnarray}}
\newcommand{\bal}{\begin{align}}
\newcommand{\eal}{\end{align}}
\newcommand{\ba}{\begin{array}}
\newcommand{\ea}{\end{array}}
\newcommand{\nn}{\nonumber}
\newcommand{\Mod}[1]{\,\,\left(\operatorname{mod}\, #1\right)}

\newcommand{\abs}[1]{\left| #1 \right|}

\DeclareFontFamily{U}{wncy}{}
\DeclareFontShape{U}{wncy}{m}{n}{<->wncyr10}{}
\DeclareSymbolFont{mcy}{U}{wncy}{m}{n}
\DeclareMathSymbol{\Sha}{\mathord}{mcy}{"58}

\DeclareMathOperator{\Cl}{Cl}
\DeclareMathOperator{\PEC}{PEC}

\DeclareMathOperator{\EC}{EC}

\DeclareMathOperator{\EU}{EU}

\newcommand{\fp}{\mathfrak{p}}
\DeclareMathOperator{\Art}{Art}

\newcommand{\mm}{\mathfrak{m}}
\newcommand{\ct}{\dagger}

\begin{document}

\title[SIC-POVMs and the Stark conjectures]{SIC-POVMs and the Stark conjectures}
\author{Gene S. Kopp}
\address{School of Mathematics, University of Bristol, Bristol, UK and Heilbronn Institute for Mathematical Research, Bristol, UK}
\email{gene.kopp@bristol.ac.uk}



\date{\today}

\thanks{This work was partially supported by NSF grant DMS-1401224, NSF grant DMS-1701576, and NSF RTG grant 0943832.}

\begin{abstract} 
The existence of a set of $d^2$ pairwise equiangular complex lines (equivalently, a SIC-POVM) in $d$-dimensional Hilbert space is currently known only for a finite set of dimensions $d$.  We prove that, if there exists a set of real units in a certain ray class field (depending on $d$) satisfying certain congruence conditions and algebraic properties, a SIC-POVM may be constructed when $d$ is an odd prime congruent to 2 modulo 3. We give an explicit analytic formula that we expect to yield such a set of units.
Our construction uses values of derivatives of zeta functions at $s=0$ and is closely connected to the Stark conjectures over real quadratic fields.

We verify numerically that our construction yields SIC-POVMs in dimensions 5, 11, 17, and 23, and we give the first exact solution to the SIC-POVM problem in dimension 23.
\end{abstract}

\maketitle


\section{Introduction}

A set of $m$ \textit{complex equiangular lines} in $d$ dimensions is a set of one-dimensional subspaces $\C v_1,\C v_2,\ldots, \C v_{m}$ in $\C^d$ having equal angles $\arccos\left(\frac{\abs{\langle v_i, v_j \rangle}}{\parallel v_i \parallel \cdot \parallel v_j \parallel}\right) = \theta$ for all $i \neq j$. These configurations were first studied in the context of design theory in the 1970s. The maximal cardinality of a set of complex equiangular lines was shown to be bounded above by $d^2$ by Delsarte, Goethals, and Seidel \cite{delsarte}.

A set of complex equiangular lines achieving this upper bound---$d^2$ lines in dimension $d$---is equivalent to a set of quantum measurements known as a SIC-POVM (symmetric informationally complete positive operator-valued measure).
SIC-POVMs were introduced in 1999 by Zauner \cite{zauner,zaunertrans} and have applications to quantum information processing (e.g., \cite{power,detection}). Their existence has implications for quantum foundations due to their presence in the theory of quantum Bayesianism (or QBism) \cite{qbism}.

SIC-POVMs are conjectured to exist in all dimensions. They have been proven to exist (by explicit construction) in dimensions 1--21, 24, 28, 30, 31, 35, 37, 39, 43, 48, 124, and 323 \cite{zauner,G04,appleby4,G05,G08,scott1,appleby5,appleby6,applebyconstructing,scott3}, and approximate numerical solutions have been found in every dimension up to 151 and several other dimensions up to 844 \cite{renes,scott1,scott2,play,scott3}. (Some further unpublished solutions have been found by Grassl and Scott: An exact solution in dimension 53 and numerical solutions in all dimensions up to 165, dimension 1155, and dimension 2208 \cite{grasslcorr}.)
All but one (the Hoggar lines \cite{hoggar2}) of the known SIC-POVMs are (up to unitary symmetry) so-called Heisenberg SIC-POVMs---the orbit of a single vector under the action of a discrete Heisenberg group.

Appleby, Flammia, McConnell, and Yard \cite{appleby1,appleby2} recently discovered an empirical connection between SIC-POVMs and Hilbert's $12$th problem for real quadratic fields. Hilbert's $12$th problem is a longstanding open problem in number theory. For a base field $K$, it asks for an explicit construction of the abelian extensions of $K$, analogous to the Kronecker-Weber theorem in the case when $K=\Q$ and the theory of complex multiplication in the case when $K$ is imaginary quadratic. 

Appleby et. al. \cite{appleby1,appleby2} observe that, for $d > 3$, the known examples of Heisenberg SIC-POVMs are always defined over abelian extensions of the real quadratic field $\Q(\sqrt{(d+1)(d-3)})$. They predict the existence of a certain Galois orbit of SIC-POVMs called a \textit{minimal multiplet}, and they conjecturally identify the specific ray class field over which the SIC-POVMs in the minimal multiplet should be defined.

This paper formulates a conjectural construction of a Heisenberg SIC-POVM in infinitely many dimensions $d=5, 11, 17, 23, 29, 41, \ldots$, which are those $d$ that are odd prime numbers congruent to $2$ modulo $3$. 
We split our construction into two pieces. The first is \Cref{conj:family}, which predicts the existence of a Galois orbit of real algebraic units in a certain class field over $\Q(\sqrt{(d+1)(d-3)})$ satisfying several strong conditions. We prove in \Cref{thm:main} that such a set of units (if it exists) may be used to construct a SIC-POVM in dimension $d$. The second piece of the construction is \Cref{conj:unitsL}, which gives an analytic formula for a set of real numbers that we expect to be algebraic units satisfying the conditions of \Cref{conj:family}.

Of crucial importance to our construction in \Cref{conj:unitsL} are the conjectures of Stark on the leading terms of the Taylor expansions of Hecke and Artin $L$-functions at $s=1$ and $s=0$ \cite{stark1,stark2,stark3,starkproc,stark4}. Stark made his conjectures in the 1970s, and they remain open over every base field except for $\Q$ (known to Dirichlet) and imaginary quadratic fields (proof due to Stark \cite{stark1}). Much work has been done toward attacking or refining the Stark conjecture; one vital reference is the proceeding of the 2002 ``International Conference on Stark's Conjectures and Related Topics'' at Johns Hopkins University \cite{burns2004stark}.

Specifically, we use the rank 1 abelian Stark conjecture in the real quadratic case \cite{stark3,starkproc}. A key observation made in this paper is that the overlap phases of certain SIC-POVMs are Galois conjugate to square roots of Stark units---units in ray class fields predicted by Stark to coincide with $\exp(Z_A'(0))$ for a certain differenced ray class zeta function $Z_A(s)$. A congruence condition modulo a prime lying over $(d)$ determines the signs of the square roots; see \cref{eq:sqroot}.

Our construction has been verified numerically in dimensions $d=5, 11, 17,$ and $23$. Moreover, our methods give \textit{exact} solutions in these dimensions.
Namely, we can conjecturally determine Stark units as algebraic numbers using sufficiently high-precision approximate
values of $\exp(Z_A'(0))$, and then
verify directly that they produce a SIC-POVM according to the recipe
provided by our conjectures. 
Thus, we provide the first 
exact expression for a SIC-POVM in dimension $23$.

The paper is organized as follows. In \cref{sec:backsic} and \cref{sec:backnum}, we provide mathematical background and an introduction to SIC-POVMs and the Stark conjectures, and we define the notation to be used in the rest of the paper. In \cref{sec:conjs}, we present our main results and conjectures. 
Finally, in \cref{sec:egz}, we present data verifying our main conjecture numerically in dimensions $d=5, 11, 17,$ and $23$.

This paper includes and extends work from Chapter 5 of the author's PhD thesis \cite{thesis}.

\section{Acknowledgments}

Thank you to Jeff Lagarias for many inspiring conversations about the SIC-POVM problem, Hilbert's 12th problem, and the Stark conjectures, and for helpful comments on drafts of this paper, especially regarding the formulation of the conjectures.

I thank Marcus Appleby, Neil Gillespie, Markus Grassl, Gary McConnell, and Harold Stark for fruitful discussions.

\section{Complex equiangular lines}\label{sec:backsic}

In this section, we introduce the definitions from quantum information theory and design theory that we need to state our main conjectures. Specifically, we define SIC-POVMs and Heisenberg SIC-POVMs. We discuss extended unitary equivalence and Galois equivalence of SIC-POVMs and define the notion of a multiplet of Heisenberg SIC-POVMs.

\subsection{Definition of SIC-POVMs}

The study of SIC-POVMs began with Zauner's 1999 PhD thesis \cite{zauner} (see English translation \cite{zaunertrans}). The term ``SIC-POVM'' was attached to the concept in 2004 by Renes, Blume-Kohout, Scott, and Caves \cite{renes}.

The maximal number of equiangular complex lines possible in $\C^d$ is $d^2$; this was originally proved in 1975 by Delsarte, Goethals, and Seidel using orthogonal polynomials \cite{delsarte}.
\begin{prop}[Delsarte, and Goethals, and Seidel \cite{delsarte}]\label{prop:delsarte1}
Let $\alpha > 0$.  Consider a set $V$ of unit vectors in $\C^d$ spanning equiangular lines; that is, $\abs{\langle v, w \rangle}^2 = \alpha$ whenever $v, w \in V$ and $v \neq w$.  Then, $\abs{V} \leq d^2$.
\end{prop} 
The same authors also show that, for a set of $d^2$ equiangular complex lines, the size of the angle is determined.
\begin{prop}[Delsarte, and Goethals, and Seidel \cite{delsarte}]\label{prop:delsarte2}
For any set $V$ of unit vectors in $\C^d$ spanning $d^2$ equiangular lines with $\abs{\langle v, w \rangle}^2 = \alpha$ whenever $v, w \in V$ and $v \neq w$,
\begin{equation}
\alpha = \frac{1}{d+1},
\end{equation}
and thus the common angle is $\arccos\left(\frac{1}{\sqrt{d+1}}\right)$.
\end{prop}

A SIC-POVM is defined as a set of measurements (i.e., projectors---idempotent Hermitian operators) on $d$-dimensional Hilbert space satisfying certain properties, and it is equivalent to a set of equiangular complex lines achieving the upper bound from \Cref{prop:delsarte1}.
\begin{defn}[SIC-POVM]
A \textbf{symmetric informationally complete positive operator-valued measure (SIC-POVM)} is a set $\{\Pi_1, \Pi_2, \ldots, \Pi_{d^2}\}$ of 
cardinality $d^2$ consisting of  rank 1 Hermitian $d \times d$ matrices $\Pi_i$ such that each $\Pi_i^2 = \Pi_i$, and
\begin{equation}
\Tr\left(\Pi_i\Pi_j\right) = \left\{\begin{array}{ll}
1 & \mbox{ if } i=j, \\
\frac{1}{d+1} & \mbox{ if } i\neq j.
\end{array}\right.
\end{equation}
Write $\Pi_i = v_i v_i^\ct$, where $v_i$ is a column vector, and $v_i^\ct$ is its conjugate-transpose, a row vector. Then, the $v_i$ define a set of $d^2$ equiangular complex lines in $\C^d$, and conversely any set of $d^2$ equiangular complex lines in $\C^d$ define a SIC-POVM.
We will use the term ``SIC-POVM'' interchangeably with ``set of $d^2$ equiangular complex lines in $\C^d$.''
\end{defn}

There are two types of operators on $\C^d$ preserving the SIC-POVM property.
A SIC-POVM $\{\C v_1,\ldots,\C v_{d^2}\}$ may be ``rotated'' by any unitary matrix $U \in \U(d) = \{U \in \GL(\C^d) : UU^\ct = 1\}$ to obtain another SIC-POVM $\{\C Uv_1,\ldots,\C Uv_{d^2}\}$.  Moreover, if $C_d$ is the complex conjugation operator on $\C^d$, so that $C_d v := \ol{v}$, then $C_d$ also preserves the SIC-POVM property (and the same holds for any ``antiunitary'' operator of the form $C_d U$).  These operators may be collected together to form the extended unitary group.
\begin{defn}
Define the \textbf{extended unitary group} $\EU(d) := \U(d) \sqcup C_d \U(d)$.
\end{defn}
\begin{lem}
The action of $\EU(d)$ takes SIC-POVMs to SIC-POVMs.
\end{lem}
\begin{proof}
The action of $\U(d)$ preserves the Hermitian inner product; $C_d \U(d)$ conjugates the Hermitian inner product.  Thus, both preserve its absolute value and thus the SIC-POVM property.
\end{proof}

\begin{defn}[Overlaps]
If $v_1, v_2, \ldots, v_{d^2}$ are unit vectors in $\C^d$ defining a SIC-POVM, the \textbf{overlaps} are the $d^4$ complex numbers $\langle v_i, v_j \rangle$. The \textbf{overlap phases}, for $i \neq j$, are the complex numbers $\sqrt{d+1}\langle v_i, v_j \rangle$, which lie on the unit circle.
\end{defn}

\subsection{Definition of Heisenberg SIC-POVMs}
Heisenberg SIC-POVMs are a special class of SIC-POVMs.
Let $\zeta_d = e\left(\frac{1}{d}\right) = \exp\left(\frac{2\pi i}{d}\right)$ be a $d$th root of unity.
\begin{defn}[Heisenberg group]
Let $d' = d$ if $d$ is odd, $d' = 2d$ if $d$ is even.  Let $I$ be the $d \times d$ identity matrix.
The Heisenberg group $\H(d)$ is the finite group of order $d'd^2$ generated by the $d \times d$ scalar matrix $\zeta_{d'} I$ and the $d \times d$ matrices
\begin{equation}
X = \left(\begin{array}{ccccc}
0 & 0 & \cdots & 0 & 1 \\
1 & 0 & \cdots & 0 & 0 \\
0 & 1 & \cdots & 0 & 0 \\
\vdots & \vdots & \ddots & \vdots & \vdots \\
0 & 0 & \cdots & 1 & 0
\end{array}\right), \ \ 
Z = \left(\begin{array}{ccccc}
1 & 0 & 0 & \cdots & 0 \\
0 & \zeta_d & 0 & \cdots & 0 \\
0 & 0 & \zeta_d^2 & \cdots & 0 \\
\vdots & \vdots & \vdots & \ddots & \vdots \\
0 & 0 & 0 & \cdots & \zeta_d^{d-1}
\end{array}\right).
\end{equation}
\end{defn}
The Heisenberg group $\H(d)$ spans the vector space $\M_d(\C)$ of $d \times d$ complex matrices, and a canonical basis is given as follows.
\begin{defn}[Heisenberg basis]
Let $\tau_d = \exp\left(\frac{(d+1)\pi i}{d}\right) = -\exp\left(\frac{\pi i}{d}\right)$.
The set of $d^2$ matrices 
\begin{equation}\label{eq:Dmn}
D_{m,n} = \tau_{d}^{mn}X^m Z^n \mbox{ for } 0 \leq m,n \leq d-1
\end{equation}
forms a basis of $\M_d(\C)$ over $\C$, and this basis is called the \textbf{Heisenberg basis}. (The $D_{m,n}$ are known as \textbf{displacement operators}.)
\end{defn}
Empirically, all but one of the known SIC-POVMs are equivalent to orbits of the Heisenberg group action, and this observation motivates the following definition.
\begin{defn}[Heisenberg SIC-POVM]
A \textbf{Heisenberg SIC-POVM} is a SIC-POVM of the form $\{\C D_{m,n}v : 0 \leq m,n \leq d-1\}$ for some vector $v \in \C^d$.  This $v$ is called a \textbf{fiducial vector}.
\end{defn}

The elements of $\EU(d)$ that preserve the property of being a \textit{Heisenberg} SIC-POVM are restricted to a finite group, the \textbf{extended Clifford group} $\EC(d)$, defined to be the normalizer of $\H(d)$ inside $\EU(d)$.
\begin{lem}
If $v$ is a fiducial vector for a Heisenberg SIC-POVM, and $\g \in \EC(d)$, then $\g v$ is also a fiducial vector for a Heisenberg SIC-POVM.  
Conversely, if $v$ and $w$ are $\EU(d)$-equivalent fiducial vectors, they are in fact $\EC(d)$-equivalent.
\end{lem}
\begin{proof}
See Scott and Grassl \cite{scott1}, section 3.
\end{proof}

Some authors consider the larger class of \textbf{group covariant SIC-POVMs}, those which are orbits of some subgroup of $\U(d)$. 
The Hoggar lines are group convariant for $\H(2) \tensor \H(2) \tensor \H(2)$, so all known SIC-POVMs are group convariant. In the case of prime dimension, it was shown by Zhu that all group covariant SIC-POVMs are equivalent to Heisenberg SIC-POVMs \cite{primezhu}.

\subsection{The Galois action of Heisenberg SICs}
In 2016, Appleby, Flammia, McConnell, and Yard \cite{appleby1,appleby2} numerically discovered a surprising connection between SIC-POVMs and ray class fields of real quadratic fields.  For all Heisenberg SIC-POVMs they checked, they found that the ratios of the entries of the fiducial vector lie in an abelian extension of the real quadratic field $\Q(\sqrt{\Delta})$, where $\Delta = \Delta_d = (d+1)(d-3)$.  The field $\Q(\sqrt{\Delta})$ is special because it contains a small unit 
\begin{equation}
\e = \e_d = \frac{(d-1)+\sqrt{\Delta}}{2}.
\end{equation}
The unit $\e$ has the property that $\e^3 \con 1 \Mod{d}$, and its presence is related to the order 3 ``Zauner symmetry'' enjoyed by many known SIC-POVMs. We call $\e$ the \textbf{Zauner unit}. Note that the Zauner unit is not always equal to the fundamental unit, but is sometimes a higher power of it; for example, $\e_{15} = (2+\sqrt{3})^2$. 

Let $K = \Q(\sqrt{\Delta})$, and let $E$ be the field generated by the ratios of the entries of the fiducial vector along with the $d'$th roots of unity, where $d' = d$ if $d$ is odd, and $d' = 2d$ if $d$ is even.  
If $v$ is a Heienberg fiducial vector and $\s \in \Gal(E/K)$, then $v^\s$ is also a Heisenberg fiducial vector; $v^\s$ may or may not lie in the same $\EC(d)$ orbit as $v$.  This Galois action respects orbits because $(\g v)^\s = \g^\s v^\s$ and $\EC(d)$ is Galois-closed.  
\begin{defn}[Multiplet]
The set of all those $\EC(d)$-orbits of fiducial vectors, which are Galois equivalent to a given $\EC(d)$-orbit, is called a \textbf{multiplet}.
\end{defn}

\section{Explicit class field theory and zeta functions}\label{sec:backnum}

In this section, we give the number-theoretic background we need to state our conjectures.
A much more complete exposition of class field theory may be found in Neukirch's book \cite{neukirch}.
The real quadratic Stark conjectures may be found in \cite{stark3}.

\subsection{Global class field theory}

Let $K$ be a number field, and let $\OO_K$ be its ring of algebraic integers, the maximal order of $K$.

A modulus $\mm$ is a pair $\mm = (\cc, S)$, where $\cc$ is an ideal of $\OO_K$, and $S$ is a subset of the (possibly empty) set of real embeddings $K \to \R$. Associate to the real embeddings $\rho_1, \ldots, \rho_r$ the ``infinite primes'' $\infty_1, \ldots, \infty_r$, and write $\mm$ using the notation $\mm = \cc \prod_{\rho_j \in S} \infty_j$. (An example of this notation is $\mm = (7)\infty_1\infty_3$).

\begin{defn}\label{defn:ray}
If $\mm = (\cc,S)$, define the \textbf{ray class group modulo $\mm$} to be
\begin{equation}
\Cl_{\mm} = \frac{\{\mbox{fractional ideals coprime to } \cc\}}{\{\mbox{principal frac. ideals } (\alpha) \mbox{ coprime to } \cc \mbox{ with } \alpha \con 1 \Mod{\cc} \mbox{ and } \rho(a) > 0 \mbox{ for all } \rho \in S\}}.
\end{equation}
\end{defn}

If $\cc'$ is any nonzero subideal of $\cc$, then the ray class group is unaffected by imposing the stronger condition of coprimility to $\cc'$. 
\begin{equation}
\Cl_{\mm} \isom \frac{\{\mbox{fractional ideals coprime to } \cc'\}}{\{\mbox{principal frac. ideals } (\alpha) \mbox{ coprime to } \cc' \mbox{ with } \alpha \con 1 \Mod{\cc} \mbox{ and } \rho(a) > 0 \mbox{ for all } \rho \in S\}}.
\end{equation}

Consider two moduli $\mm = (\cc,S)$ and $\mm' = (\cc',S')$. We say that $\mm|\mm'$ if $\cc \supset \cc'$ and $S \subset S'$. The quotient maps $\pi_{\mm',\mm} : \Cl_{\mm'} \surj \Cl_{\mm}$ are defined by first imposing comprimality to $\cc'$ in $\Cl_{\mm}$, then modding out by the stronger congruence and positivity conditions modulo $\mm'$.

The abelian Galois extensions of $K$ are associated to quotients of ray class groups, by the following existence theorem of Takagi.
\begin{thm}[Existence theorem]\label{thm:existence}
Let $K$ be a number field, and let $K^{\rm ab}$ be the maximal abelian extension of $K$ (an infinite-degree extension). Then, there is a natural isomorphism of the Galois group $\Gal(K^{\rm ab}/K)$ with the inverse limit $\ds\lim_{\leftarrow} \Cl_{\mm}$ with respect to the quotient maps $\pi_{\mm',\mm}$. This isomorphism is called the \textbf{Artin map}:
\begin{equation}
\Art :  \ds\lim_{\leftarrow} \Cl_{\mm} \to \Gal(K^{\rm ab}/K).
\end{equation}
The field $L_\mm/K$ corresponding to $\Cl_{\mm}$ by Galois theory---so that $\Gal(L_\mm/K) \isom \Cl_{\mm}$ under the Artin map---is called the \textbf{ray class field} of $K$ modulo $\mm$. If $\mm = (\cc,S)$, then the extension $L_\mm$ is ramified only at the primes dividing $\cc$ and the real places in $S$.
\end{thm}
\begin{proof}
See \cite{neukirch}, Chapter VI, especially Theorem 7.1.
\end{proof}

\subsection{Ray class zeta functions and Hecke $L$-functions}

We now define two Dirichlet series, $\zeta_A(s)$ and $Z_A(s)$, attached to a ray ideal class $A$ of the ring of integers of a number field.
\begin{defn}[Ray class zeta function]
Let $K$ be any number field, and let $\cc$ be an ideal of the maximal order $\OO_K$.  Let $S$ be a subset of the real places of $K$ (i.e., the embeddings $K \inj \R$).  Let $A$ be a ray ideal class modulo $\mm = (\cc,S)$. 
Define the \textbf{zeta function of the ray class $A$} to be
\begin{equation}
\zeta(s,A) = \sum_{\aa \in A} N(\aa)^{-s} \mbox{ for } \Re(s) > 1.
\end{equation}
\end{defn}
This function may be meromorphically continued to the whole complex plane.
\begin{thm}
This function $\zeta(s,A)$ has a meromorphic continuation to $\C$. It has a simple pole at $s=1$ with residue independent of $A$, and no other poles.
\end{thm}
\begin{proof}
See Neukirch \cite{neukirch}, Chapter VII, Theorem 8.5 for the corresponding statement about Hecke $L$-functions, from which this theorem follows.
\end{proof}
The pole at $s=1$ may be eliminated by considering the function $Z_A(s)$, defined as follows. The function $Z_A(s)$ is holomorphic everywhere.
\begin{defn}[Differenced ray class zeta function]
Let $R$ be the element of $C_{\mm}$ defined by 
\begin{equation}
R= \{a\OO_K : a \con -1 \Mod{\cc} \mbox{ and $a$ is positive at each place in $S$}\}.
\end{equation}
Define the \textbf{differenced zeta function of the ray class $A$} to be
\begin{equation}
Z_A(s) = \zeta(s,A) - \zeta(s,RA).
\end{equation}
\end{defn}

Hecke $L$-functions (of finite-order Hecke characters) are linear combinations of ray class zeta functions and are ubiquitous in modern number theory, largely because they have an Euler product.  Conversely, ray class zeta functions may be expressed as linear combinations of Hecke $L$-functions. 
\begin{defn}[$L$-function of a finite-order Hecke character]\label{defn:hecke}
Let $K$ be a number field and $\mm$ a modulus of $K$. Let $\chi : \Cl_\mm \to \C^\times$ be a group homomorphism. The Hecke $L$-function is
\begin{equation}
L(s,\chi) = \sum_{A \in \Cl_\mm} \chi(A) \zeta(s,A).
\end{equation}
\end{defn}
All our results and conjectures are stated using (differenced) ray class zeta functions rather than Hecke $L$-functions, as they are cleaner that way.
However, our computer calculations (see \cref{sec:egz}) rely on Magma's built-in algorithms for computing Hecke $L$-functions.

\subsection{A Stark conjecture over real quadratic base field}

The Existence Theorem \ref{thm:existence} does not provide a procedure for actually building the ray class field $L_\mm$.
Explicit constructions are known when the base field $K$ is $\Q$ or an imaginary quadratic field; however, it is not known how to construct ray class fields explicitly in general.

The Stark conjectures provide one approach to developing an explicit class field theory. The following conjecture, due to Stark \cite{stark3}, gives a conjectural generator for a ray class field over a real quadratic field, under certain conditions.

\begin{conj}[Stark conjecture, rank 1 real quadratic case]\label{conj:stark}
Let $K$ be a real quadratic number field with real embeddings $\rho_1$ and $\rho_2$.
Let $\cc$ be a nonzero ideal of the ring of integers of $K$ with the property that, if $\e \in \OO_K^\times$ such that $\e \con 1 \Mod{\cc}$, then $\rho_1(\e) > 0$.
Let $A$ be a ray ideal class in $\Cl_{\cc \infty_2}$.  Let $L_{\cc \infty_j}$ be the ray class field of $K$ modulo $\cc \infty_j$, and let $\tilde\rho_j$ be a choice of embedding of $L_{\cc \infty_j}$ extending $\rho_j$. (Here, $\tilde\rho_1(L_{\cc \infty_2})=\tilde\rho_2(L_{\cc \infty_1})$ is a real field, and $\tilde\rho_1(L_{\cc \infty_1})=\tilde\rho_2(L_{\cc \infty_2})$ is a complex [non-real] field.)
Then,
\begin{itemize}
\item[(1)] $Z_{A}'(0) = \log(\tilde\rho_1(\alpha_A))$ for some real algebraic unit $\alpha_A \in L_{\cc \infty_2}$, and furthermore $L_{\cc \infty_2} = K(\alpha_A)$.
\item[(2)] The units $\alpha_A$ are compatible with the Artin map $\Art : \Cl_{\cc \infty_2} \to \Gal\left(L_{\cc \infty_2}/K\right)$.  Specifically, $\alpha_A = \alpha_I^{\Art(A)}$, where $I \in \Cl_{\cc \infty_2}$ is the identity class.
\end{itemize}
\end{conj}

\section{Towards an infinite family of Heisenberg SIC-POVMs}\label{sec:conjs}

We now state our main conjectures and results. Our first conjecture predicts the existence of a Galois orbit of real algebraic units in the ray class field $L_{(d)\infty_2}$ satisfying certain strong conditions, for $d$ an odd prime such that $d \con 2 \Mod{3}$. We will show in \Cref{thm:main} that these conditions imply the existence of a Heisenberg SIC-POVM in dimension $d$.

\begin{conj}\label{conj:family}
Let $d$ be an odd prime such that $d \con 2 \Mod{3}$. Let $\Delta = (d+1)(d-3)$ and $K = \Q(\sqrt{\Delta})$, 
and consider the class group $\Cl_{(d)\infty_2}$. 
With indices $m,n \in \Z/d\Z$, $(m,n) \neq (0,0)$, let 
\begin{equation}
A_{m,n} = \{\alpha\OO_K : \alpha \con m+n\sqrt{\Delta} \Mod{d} \mbox{ and } \rho_2(\alpha)>0\} \in \Cl_{(d)\infty_2},
\end{equation}
and let $\Art : \Cl_{(d) \infty_2} \to \Gal\left(L_{(d) \infty_2}/K\right)$ denote the Artin map of class field theory.
Then, there is a real algebraic unit $\alpha$ such that the ray class field $L_{(d)\infty_2} = K(\alpha)$ and such that its (real) Galois conjugates $\alpha_{m,n} := \alpha^{\Art(A_{m,n})}$ over the Hilbert class field $L_{(1)}$ have the following properties.
\begin{itemize}
\item[(1)] $\alpha_{-m,-n} = \alpha_{m,n}^{-1}$.
\item[(2)] The $\alpha_{m,n} \con 1 \Mod{\fp}$ for all prime ideals $\fp$  of $\OO_{L_{(d)\infty_2}}$ dividing $d\OO_{L_{(d)\infty_2}}$.
\item[(3)] The roots of $(d+1)x^2 = \alpha_{m,n}$ are in $L_{(d)\infty_2}$.
\item[(4)] Fix a choice of $\fp$  dividing $d\OO_{L_{(d)\infty_2}}$, and let $\nu_{m,n}$ be the unique root of 
\begin{equation}\label{eq:sqroot} 
(d+1)x^2 = \alpha_{m,n}
\end{equation}
satisfying $\nu_{m,n} \con 1 \Mod{\fp}$. Let $\nu_{0,0} = 1$. Then, the matrix
\begin{equation}\label{eq:M}
M = \frac{1}{d} \sum_{m=0}^{d-1} \sum_{n=0}^{d-1} \nu_{m,n} D_{-m,-n} 
\end{equation}
is idempotent (that is, $M^2 = M$) and rank 1 (that is, all the rank 2 minors of $M$ vanish). (See \cref{eq:Dmn} for the definition of $D_{m,n}$.)
\end{itemize}
\end{conj}

The condition that $d$ is prime and $d \con 2 \Mod{3}$ is necessary because \cref{eq:Amn} does not make sense for all $(m,n) \neq (0,0)$ without it. For general $d$, the principal subgroup of $\Cl_{(d)\infty_2}$ is represented by those $A_{m,n}$ such that $(m+n\sqrt{\Delta})$ is relatively prime to $(d)$. The condition that $d$ is prime and $d \con 2 \Mod{3}$ is equivalent to $(d)$ being prime in $\OO_K$, that is, to $(m+n\sqrt{\Delta})$ being relatively prime to $(d)$ whenever $(m,n) \con (0,0) \Mod{d}$.

Our second conjecture provides an explicit analytic formula for a set of real numbers that we expect to be Galois conjugate algebraic units with the desired properties. (The $Z_{A}'(0)$ are known to be real numbers.)
\begin{conj}\label{conj:unitsL}
A unit $\alpha$ satisfying \Cref{conj:family} and its Galois conjugates over $K$ may be constructed as Stark units
\begin{equation}\label{eq:alpha}
\alpha^{\Art(A)} = \exp\left(Z_{A}'(0)\right),
\end{equation}
for all $A \in \Cl_{(d)\infty_2}$.
\end{conj}

\Cref{conj:stark}, due to Stark, predicts that $\exp\left(Z_{A}'(0)\right) = \alpha^{\Art(A)}$ for an algebraic unit $\alpha$ generating $L_{(d)\infty_2}$ over $K$. Point (1) of \Cref{conj:family} follows, because $Z_{RA}(s) = - Z_A(s)$. Points (2), (3), and (4) of \Cref{conj:family} provide, to our knowledge, new predictions about Stark units.

There are several analytic formulas for the derivative zeta values $Z_{A}'(0)$ appearing on the right-hand side of \cref{eq:alpha}.
Shintani's Kronecker limit formula expresses $\exp\left(Z_{A}'(0)\right)$ as a product of special values of Barnes's double gamma function \cite{shintanik,shintanicertain}.
A different formula appears in Chapter 4 of the author's PhD thesis \cite{thesis}.

We state \Cref{conj:family} and \Cref{conj:unitsL} as two separate conjectures due to the possibility that it may be possible to prove the first without proving the second. One route by which this might be done is through the use of $p$-adic zeta functions rather than Archimedian zeta functions. The $p$-adic approach seems hopeful because $p$-adic analogues of the Stark conjectures have been proven in some cases. 
At least, the $\fp$-adic condition (2) for corresponding $d$-adic special values looks potentially amenable to proof.

We now prove a lemma about the structure of the class fields over $K$, which we require for our main result, \Cref{thm:main}.

\begin{lem}\label{lem:minusunits}
Let $K = \Q(\sqrt{\Delta})$, where $\Delta = (d+1)(d-3)$ and $d \geq 4$. Let $\e_d = \frac{(d-1)+\sqrt{\Delta}}{2}$.
If $\eta \in \OO_K^\times$ and $\eta \con 1 \Mod{d}$, then $\eta = \e_d^{3k}$ for some $k \in \Z$, and in particular $\e_d$ is totally positive.
It follows that the class fields $L_{(d)}$, $L_{(d)\infty_1}$, $L_{(d)\infty_2}$, and $L_{(d)\infty_1\infty_2}$ are all distinct.
\end{lem}
\begin{proof}
It suffices to prove that the minimal $\eta > 1$ in $\OO_K$ such that $\eta \con 1 \Mod{d}$ is $\e^3$; consider the minimal such $\eta$. We have $\eta^n = \e_d^3$ for some $n \in \Z$, $n > 0$.

Let $\eta'$ and $\e'$ denote the nontrivial Galois conjugates of $\eta$ and $\e$, respectively.
If $n \geq 3$, then $\eta + \eta' \leq \e + \e' = d-1$, so it's impossible to have $\eta \con 1 \Mod{d}$.

Thus, $n \leq 2$. 
Suppose $n = 2$.
Then $\e_d$ has a square root is $K$; it's straightforward to check that this happens exactly when $d+1$ is a square, in which case $\e_d = \e_{d_0}^2$ with $d = d_0^2-2d_0$.
Then
\begin{align}
\eta &= \e_d^3 \\
&= \frac{(d_0-1)(d-2)+d\sqrt{\Delta}}{2} \\
&\con -d_0 \Mod {d}.
\end{align}
Thus, $\eta \not\equiv 1 \Mod{d}$.

So we must have $n = 1$, which is what we wanted to prove.

It follows that the class groups $\Cl_{(d)}$, $\Cl_{(d)\infty_1}$, $\Cl_{(d)\infty_2}$, and $\Cl_{(d)\infty_1\infty_2}$ are all distinct.
Thus, by \Cref{thm:existence}, the class fields $L_{(d)}$, $L_{(d)\infty_1}$, $L_{(d)\infty_2}$, and $L_{(d)\infty_1\infty_2}$ are all distinct.
\end{proof}

We are now ready to prove that \Cref{conj:family} implies the existence of a Heisenberg SIC-POVM.

\begin{thm}\label{thm:main}
Let $d$ be an odd prime such that $d \con 2 \Mod{3}$.
Assume \Cref{conj:family}, and let $M$ be the matrix constructed therein.
Let $\sigma \in \Gal(L_{(d)\infty_2}/\Q)$ be any Galois automorphism \textit{not} fixing $K$; that is, $\sigma(\sqrt{\Delta}) = -\sqrt{\Delta}$. Then $\sigma(M) = vv^\ct$ for a fiducial vector $v$ of a Heisenberg SIC-POVM.
\end{thm}
\begin{proof}
The $\alpha_{m,n}$ are Galois conjugate real algebraic units that generate $L_{(d)\infty_2}$ over $K$. 
By \Cref{lem:minusunits}, $L_{(d)} \neq L_{(d)\infty_2}$.
The $\alpha_{m,n}$ must not be totally real, because if they were, they would all lie in $L_{(d)}$. 
Moreover, by (1), $\alpha_{-m,-n} = \alpha_{m,n}^{-1}$, so the $\alpha_{m,n}$ are Galois conjugate to their inverses. 
Take $\tau \in \Gal(L_{(d)\infty_2}/K)$ such that $\tau(\alpha) = \alpha^{-1}$. Then, $\tau(\alpha_{m,n}) = \alpha_{m,n}^{-1}$ because $\Gal(L_{(d)\infty_2}/K)$ is abelian. Thus,
\begin{equation}
(\sigma\tau\sigma^{-1})(\sigma(\alpha_{m,n})) = \sigma(\alpha_{m,n})^{-1}.
\end{equation}

We have $\tau = \Art_R$, where 
\begin{equation}
R = \{\beta\OO_K : \beta\con 1 \Mod{d} \mbox{ and } \rho_2(\beta)>0 \} \in \Cl_{(d)\infty_2}.
\end{equation}
Thus, by class field theory, the fixed field of $\tau$ is $L_{(d)}$. It follows that the class field of $\sigma \tau \sigma^{-1} \in \Gal(L_{(d)\infty_1}/K)$ is also $L_{(d)}$. Thus, $\sigma \tau \sigma^{-1}$ must act by complex conjugation, because $L_{(d)}$ is the real subfield of the complex field $L_{(d)\infty_1}$.

Thus, $\sigma(\alpha_{m,n})$ lies on the unit circle. It follows that $\abs{\sigma(\nu_{m,n})} = \frac{1}{\sqrt{d+1}}$.
From the definition of $M$, we obtain $\Tr(MD_{m,n}) = \nu_{m,n}$. Thus,
\begin{align}
\abs{\Tr\left(\sigma(M)\sigma(D_{m,n})\right)}^2 &= \Tr\left(\sigma(M)\sigma(D_{m,n})\right)(\sigma\tau\sigma^{-1})\left(\Tr\left(\sigma(M)\sigma(D_{m,n})\right)\right) \\
&= \sigma\left(\Tr\left(MD_{m,n}\right)\tau\left(\Tr\left(MD_{m,n}\right)\right)\right) \\
&= \sigma\left(\nu_{m,n}\nu_{-m,-n}\right) \\
&= \frac{1}{d+1}.
\end{align}
Moreover, the action of $\sigma$ on the Heisenberg group is determined by its action of $\Q(\zeta_d)$, and there is some $\lambda \in \left(\Z/d\Z\right)^\times$ such that $\sigma(D_{m,n}) = \sigma(D_{m,\lambda n})$. So, by changing the value of $n$, we have
\begin{equation}\label{eq:fiddish}
\abs{\Tr\left(\sigma(M)D_{m,n}\right)}^2 = \frac{1}{d+1},
\end{equation}
for all $(m,n) \neq (0,0)$. Now write $\sigma(M) = vw^\ct$ for some $v,w \in \C^d$ with $w^\ct v = 1$, which is possible because $M$, and thus $\sigma(M)$, is rank 1 and idempotent. We will show that $\sigma(M)$ is in fact a Hermitian projector; that is, $w=v$. Conjugation-transposition acts on $\sigma(M)$ as follows:
\begin{equation}
\sigma(M)^\ct = \sum_m\sum_n \ol{\sigma(\nu_{m,n})}D_{-m,-\lambda n}^\ct
= \sum_m\sum_n \sigma(\nu_{-m,-n})D_{m,\lambda n}^\ct = \sigma(M).
\end{equation}
Thus, $w = v$. So \cref{eq:fiddish} may be rewritten as $\langle v, D_{m,n}v \rangle = \frac{1}{\sqrt{d+1}}$ for $(m,n) \neq (0,0)$; also, $\langle v, v \rangle = 1$. In other words, $v$ is a fiducial vector for a Heisenberg SIC-POVM in $\C^d$.
\end{proof}

\Cref{conj:family} has been verified numerically for $d=5$, $11$, $17$, and $23$. We discuss each case in turn in \cref{sec:egz}.

\section{Examples}\label{sec:egz}

In this section, we use \Cref{conj:family} and \Cref{conj:unitsL} to compute exact SIC-POVMs in dimensions $d = 5, 11, 17,$ and $23$. In each case, our solutions represent the minimal multiplet (in the sense of Appleby et. al. \cite{appleby2}), or (in the case of $d=23$) we expect them to. Originally, the case $d=5$ is due to Zauner \cite{zauner,zaunertrans}, $d=11$ is due to Scott and Grassl \cite{scott1}, and $d=17$ is due to Appleby et. al. \cite{applebyconstructing}. The case $d=23$ is new.

We numerically compute the numbers $\alpha_A^{\rm approx} := \exp(Z_A'(0))$ for the ideal classes $A \in \Cl_{(d)\infty_2}(K)$, using Magma's built-in functions for evaluation of Hecke L-functions. 
The $\alpha_A^{\rm approx}$ are expected to be Galois conjugate algebraic numbers (see \Cref{conj:stark}); in each case, we find that they indeed agree to high precision with Galois conjugate algebraic numbers $\alpha_A^{\rm exact}$. The $\alpha_A^{\rm exact}$ are found by using Mathematica's built-in lattice basis-reduction algorithms to compute the coefficients of a minimal polynomial $f_d(x)$ over $L_{(1)}$, the Hilbert class field of $K$.

Next, we verify conditions (2) and (3) of \Cref{conj:family} and compute the minimal polynomial $g_d(x)$ of a choice of $\nu_A$. The latter is done by factoring $f_d((d+1)x^2)$ over $L_{(1)}$ in Magma.
Finally, we verify that the $\nu_A$ satisfy condition (4) of \Cref{conj:family}.

Let $\tilde{g}_d(x)$ denote a polynomial obtained from $g_d(x)$ by applying a Galois automorphism $\sigma$ to the coefficients with the property that $\sigma\left(\sqrt{\Delta}\right) = -\sqrt{\Delta}$. The roots $\nu_{m,n}$ of $g_d(x)$ are real numbers, and we have already computed them not only as a set, but also as an orbit of the Galois group $G = \Gal(L_{(d)\infty_2}/L_{(1)})$. The roots of $\tilde{g}_d(x)$ are complex numbers of absolute value $\frac{1}{\sqrt{d+1}}$, and they are the overlaps of the SIC-POVM we are looking for. In order to compute a fiducial vector, we need to compute the Galois action on them as well.

In each case ($d= 5, 11, 17, 23$), the Galois group $G = \Gal(L_{(d)\infty_2}/L_{(1)}) \isom \left(\OO_{K}/d\OO_{K}\right)^\times/\langle \ol{\e_d} \rangle$ is a cyclic group of order $n=\frac{d-1}{3}$. Let $\gamma$ be a generator for $\left(\OO_{K}/d\OO_{K}\right)^\times$, so $\tau = \Art(A_{\gamma})$ is a generator for $G$. Fix $\nu = \nu_{1,0}$. We compute $\tau(\nu)$ in the monomial basis $\{1, \nu, \cdots \nu^{n-1}\}$ for $L_{(d)\infty_2}$ over $L_{(1)}$: 
\begin{equation}
\tau(\nu) = h_d(\nu) = \sum_{j=0}^{n-1} c_j \nu^j,
\end{equation}
where $c_j \in L_{(1)}$. This is done by solving the following linear system for the $c_j$.
\begin{equation}
\left(\begin{array}{c}
\tau(\nu) \\ \tau^2(\nu) \\ \vdots \\ \nu
\end{array}\right)
= 
\left(\begin{array}{cccc}
1 & \nu & \cdots & \nu^{n-1} \\
1 & \tau(\nu) & \cdots & \tau(\nu)^{n-1} \\
\vdots & \vdots & \ddots & \vdots \\
1 & \tau^{n-1}(\nu) & \cdots & \tau^{n-1}(\nu)^{n-1}
\end{array}\right)
\left(\begin{array}{c}
c_0 \\ c_1 \\ \vdots \\ c_{n-1}
\end{array}\right).
\end{equation}
The action on $\tau$ on the roots of $g_d(x)$ is given by applying the polynomial $h_d(x)$, so the action of a generator $\tilde{\tau}$ for $\Gal\left(L_{(d)\infty_1}/L_{(1)}\right)$ on the roots of $\tilde{g}_d(x)$ is given by applying
\begin{equation}
\tilde{h}_d(x) = \sum_{j=0}^{n-1} \tau(c_j)x^j.
\end{equation}

From the roots $\tilde{\nu}_{m,n}$ of $\tilde{g}_d(x)$, a fiducial vector $v_d$ is computed by applying $\sigma$ to \cref{eq:M}:
\begin{equation}\label{eq:Msigma}
v_d v_d^\ct = \sigma(M) = \frac{1}{d} \sum_{m=0}^{d-1} \sum_{n=0}^{d-1} \tilde{\nu}_{m,n} D_{-m,-\lambda n}.
\end{equation}
As we haven't kept track of the action of $\sigma$ on roots of unity, we must check different values of $\lambda \in \left(\Z/d\Z\right)^\times$ until we find one that works.

While this method produces an exact fiducial, we write down only a numerical fiducial in the examples that follow and in the accompanying text files. (We have already specified $v_d$ exactly up to a Galois action by writing down $g_d(x)$, and the minimal polynomials of its entries have prohibitively large coefficients in some cases.)

In the cases $d=5$, $11$, and $17$, the coefficients of $g_d(x)$ live in $K$, so $\tilde{g}_d(x)$ is determined by $g_d(x)$. In the case $d=23$, the coefficients of $g_d(x)$ live in a degree 2 extension of $K$, so there are two choices for $\tilde{g}_d(x)$. The two choices lead to two different $\PEC(23)$-orbits, 23b and 23f.

The ancillary text files accompanying this paper contain expressions for $f_d(x), g_d(x), h_d(x)$, and $v_d$; in the case $d=23$, expressions for related polynomials $g_d^\ast(x)$ and $h_d^\ast(x)$ are given instead, as explained in \cref{sec:d23}. The ancillary files are \texttt{dim5.txt}, \texttt{dim11.txt}, \texttt{dim17.txt}, \texttt{dim13.txt}, and \texttt{hstar23.txt}.

\subsection{The example $d=5$}

In dimension $d=5$, the corresponding quadratic discriminant is $\Delta = (d+1)(d-3) = 12$. The real quadratic base field is $K = \Q(\sqrt{\Delta}) = \Q(\sqrt{3})$, with Zauner unit $\e = \frac{d-1+\sqrt{\Delta}}{2} = 2+\sqrt{3}$ equal to the fundamental unit.

There is exactly one $\PEC(5)$-orbit of Heisenberg SIC-POVMs \cite{scott1}.

We construct ``the'' Heisenberg SIC-POVM in dimension $5$ explicitly using \Cref{conj:family} and \Cref{conj:unitsL}. The ray class group $\Cl_{d \infty_2}(K) \isom \Z/8\Z$. The numbers $\alpha_{A}^{\rm approx} = \exp(Z_{A}'(0))$ were computed to 50 digits of precision and were found to numerically satisfy the following polynomial $f_5(x)$ over $K$.  
\begin{align}
f_5(x) = 
x^8 &- (8 + 5\sqrt{3})x^7 + (53 + 30\sqrt{3})x^6 - (156 + 90\sqrt{3})x^5 + (225 + 130\sqrt{3})x^4\nn\\ 
   &- (156 + 90\sqrt{3})x^3 + (53 + 30\sqrt{3})x^2 - (8 + 5\sqrt{3})x + 1. \label{eq:deltapoly}
\end{align}
The minimal polynomial of the corresponding $\nu_A$ is
\begin{align}
g_5(x) = 
1296x^8 &- (648 + 1080\sqrt{3})x^7 + 648x^6 + (864 + 360\sqrt{3})x^5 - (540 + 360\sqrt{3})x^4\nn\\
   &+ (144 + 60\sqrt{3})x^3 + 18x^2 - (3 + 5\sqrt{3})x + 1.
\end{align}
The overlaps of the Heisenberg SIC are the roots of the conjugate polynomial
\begin{align}
\tilde{g}_5(x) = 
1296x^8 &- (648 - 1080\sqrt{3})x^7 + 648x^6 + (864 - 360\sqrt{3})x^5 - (540 - 360\sqrt{3})x^4\nn\\
   &+ (144 - 60\sqrt{3})x^3 + 18x^2 - (3 - 5\sqrt{3})x + 1.
\end{align}
A numerical approximation to a fiducial vector is
\begin{equation}
v_5 \approx
\left(\begin{array}{c}
0.24167903563278788347 \\
-0.42393763943145804455 - 0.23674553208033493698 i \\
\white{-}0.67406464953559540185 + 0.19581007881800632630 i \\
\white{-}0.04040992380093525849 - 0.19581007881800632630 i \\
\white{-}0.34218699574986301405 + 0.23674553208033493698 i
\end{array}\right).
\end{equation}

\subsection{The example $d=11$}

In dimension $d=11$, the corresponding quadratic discriminant is $\Delta = (d+1)(d-3) = 96$. The real quadratic base field is $K = \Q(\sqrt{\Delta}) = \Q(\sqrt{6})$, with Zauner unit $\e = \frac{d-1+\sqrt{\Delta}}{2} = 5+2\sqrt{6}$ equal to the fundamental unit.

There are three $\PEC(11)$-orbits of SIC-POVMs \cite{scott1}, denoted 11a, 11b, and 11c. They occur in two multiplets, \{11a, 11b\} and \{11c\}.

We construct the orbit 11c explicitly using \Cref{conj:family} and \Cref{conj:unitsL}. The ray class group $\Cl_{d \infty_2}(K) \isom \Z/40\Z$. The numbers $\alpha_A = \exp(Z_A'(0))$ were computed to 50 digits of precision and were found to numerically satisfy the following polynomial $f_{11}(x)$ over $K$. 
\begin{align}
f_{11}(x) =& \ (x^{40}+1) + (-106 - 44 \sqrt{6}) (x^{39}+x) + (10614 + 4334 \sqrt{6}) (x^{38}+x^2) \nn\\ &+ (-652115 - 266222 \sqrt{6}) (x^{37}+x^3) + (27825305 + 11359634 \sqrt{6}) (x^{36}+x^4) \nn\\ &+ (-877414856 - 358203120 \sqrt{6}) (x^{35}+x^5) \nn\\ &+ (21232702036 + 8668214302 \sqrt{6}) (x^{34}+x^6) \nn\\ &+ (-404105217077 - 164975264036 \sqrt{6}) (x^{33}+x^7) \nn\\ &+ (6148885983306 + 2510272190954 \sqrt{6}) (x^{32}+x^8) \nn\\ &+ (-75622522312964 - 30872765454828 \sqrt{6}) (x^9+x^{31}) \nn\\ &+ (756937617777704 + 309018488445524 \sqrt{6}) (x^{10}+x^{30}) \nn\\ &+ (-6189687857216636 - 2526929486213638 \sqrt{6}) (x^{11}+x^{29}) \nn\\ &+ (41399992237530827 + 16901476056189164 \sqrt{6}) (x^{12}+x^{28}) \nn\\ &+ (-226256732983154420 - 92368924446311502 \sqrt{6}) (x^{13}+x^{27}) \nn\\ &+ (1007258543741292244 + 411211578537502754 \sqrt{6}) (x^{14}+x^{26}) \nn\\ &+ (-3634879852543059214 - 1483933485842242424 \sqrt{6}) (x^{15}+x^{25}) \nn\\ &+ (10563883893311542549 + 4312687540103174716 \sqrt{6}) (x^{16}+x^{24}) \nn\\ &+ (-24534973105051444408 - 10016360826714109164 \sqrt{6}) (x^{17}+x^{23}) \nn\\ &+ (45162757813812803926 + 18437618670122548666 \sqrt{6}) (x^{18}+x^{22}) \nn\\ &+ (-65380387193394562674 - 26691431301568773124 \sqrt{6}) (x^{19}+x^{21}) \nn\\ &+ (74013773227204686051 + 30215996390786346646 \sqrt{6}) x^{20}. 
\end{align}
The minimal polynomial of the corresponding $\nu_A$ is
\begin{align}
g_{11}(x) =& \ (12^{20}x^{40} + 1) + (48 + 22 \sqrt{6}) (12^{19}x^{39} + x) + (1968 + 792 \sqrt{6}) (12^{18}x^{38} + x^2) 
\nn\\ &+ (25848 + 10560 \sqrt{6}) (12^{17}x^{37} + x^3) 
+ (-419472 - 171072 \sqrt{6}) (12^{16}x^{36}+x^4) 
\nn\\ &+ (-18892224 - 7714080 \sqrt{6}) (12^{15}x^{35} + x^5) 
+ (-181457280 - 74074176 \sqrt{6}) (12^{14}x^{34}+x^6) 
\nn\\ &+ (2141686656 + 874329984 \sqrt{6}) (12^{13}x^{33}+x^7)
\nn\\ &+ (62948109312 + 25698435840 \sqrt{6}) (12^{12}x^{32}+x^8) 
\nn\\ &+ (337583904768 + 137818340352 \sqrt{6}) (12^{11}x^{31}+x^9) 
\nn\\ &+ (-5182578339840 - 2115780175872 \sqrt{6}) (12^{10}x^{30}+x^{10}) 
\nn\\ &+ (-83855167709184 - 34233724293120 \sqrt{6}) (12^9x^{29}+x^{11})
\nn\\ &+ (-202894373007360 - 82831288725504 \sqrt{6}) (12^8x^{28}+x^{12})
\nn\\ &+ (4929898807885824 + 2012622741651456 \sqrt{6}) (12^7x^{27}+x^{13})
\nn\\ &+ (47257471319703552 + 19292782106492928 \sqrt{6}) (12^6x^{26}+x^{14})
\nn\\ & + (45726669189808128 + 18667833415925760 \sqrt{6}) (12^5x^{25}+x^{15})
\nn\\ &+ (-1783877902738587648 - 728265100678004736 \sqrt{6}) (12^4x^{24}+x^{16})
\nn\\ &+ (-11823467430652674048 - 4826910371186737152 \sqrt{6}) (12^3x^{23}+x^{17})
\nn\\ &+ (-11846897461773729792 - 4836475651119906816 \sqrt{6}) (12^2x^{22}+x^{18})
\nn\\ &+ (215144426763866603520 + 87832344582220677120 \sqrt{6}) (12x^{21}+x^{19})
\nn\\ &+ (1246186807345680482304 + 508753633031010385920 \sqrt{6}) x^{20}.
\end{align}
A numerical approximation to a fiducial vector is
\begin{equation}
v_{11} \approx
\left(\begin{array}{c}
0.31885501173446151953 \\
-0.00727092982813886277 - 0.27361988848296183641 i \\
-0.02661472547965484998 - 0.37021984761997660380 i \\
\white{-}0.26103622782810404567 - 0.17519492308655237763 i \\
-0.17130782441847984905 - 0.00498132225998709619 i \\
\white{-}0.115894663935244395576 + 0.023682353557509036148 \\
-0.20515008071739008667 - 0.15408741485673275946 i \\
-0.12010487681658182011 + 0.12297653825975158437 i \\
\white{-}0.24462906391022207471 - 0.16496679971724189093 i \\
\white{-}0.56644313698737356128 + 0.07879109944268139950 i \\
\white{-}0.16629395529532658791 - 0.07929717656822654842 i
\end{array}\right).
\end{equation}

\subsection{The example $d=17$}

In dimension $d=17$, the corresponding quadratic discriminant is $\Delta = (d+1)(d-3) = 252$. The real quadratic base field is $K = \Q(\sqrt{\Delta}) = \Q(\sqrt{7})$, with Zauner unit $\e = \frac{d-1+\sqrt{\Delta}}{2} = 8+3\sqrt{7}$ equal to the fundamental unit.

There are three $\PEC(17)$-orbits of SIC-POVMs \cite{scott1}, denoted 17a, 17b, and 17c. They occur in two multiplets, \{17a, 17b\} and \{17c\}.

We construct the orbit 17c explicitly using \Cref{conj:family} and \Cref{conj:unitsL}. The ray class group $\Cl_{d \infty_2}(\OO_{28}) \isom \Z/96\Z$. The numbers $\alpha_A = \exp(Z_A'(0))$ were computed to 50 digits of precision and were found to numerically satisfy the a polynomial $f_{17}(x)$ over $K$. 
Values of $f_{17}(x)$, $g_{17}(x)$, $h_{17}(x)$, and a fiducial vector $v_{17}$ may be found in the accompanying text file \texttt{dim17.txt}.

\subsection{The example $d=23$}\label{sec:d23}

In dimension $d=23$, the corresponding quadratic discriminant is $\Delta = (d+1)(d-3) = 480$. The real quadratic base field is $K = \Q(\sqrt{\Delta}) = \Q(\sqrt{30})$, with Zauner unit $\e = \frac{d-1+\sqrt{\Delta}}{2} = 11+2\sqrt{30}$ equal to the fundamental unit. Unlike our previous examples, $K$ is not a principal ideal domain, but rather has class number 2.

According to previous numerical searches, there appear to be six $\PEC(23)$-orbits of SIC-POVMs \cite{scott1}, denoted 23a, 23b, 23c, 23d, 23e, and 23f. 
We have explicitly computed one multiplet consisting of two $\PEC(23)$-orbits using \Cref{conj:family}; Markus Grassl has checked that our solutions are unitary equivalent to $\{23b, 23f\}$ \cite{grasslcorr}.
We expect, but have not shown, that the other four orbits form a multiplet \{23a, 23c, 23d, 23e\}.

The ray class group $\Cl_{d \infty_2}(K) \isom \Z/2\Z \times \Z/176\Z$. The numbers $\alpha_A^{\rm approx} = \exp(Z_A'(0))$ were computed to 500 digits of precision 
(overnight in 16 parallel threads). 
The values $\alpha_{m,n}^{\rm approx} = \exp(Z_{A_{m,n}}'(0))$ (corresponding to the principal ideal classes) were found to numerically satisfy the a polynomial $f_{23}(x)$ over the Hilbert class field $L_{(1)} = \Q(\sqrt{5},\sqrt{6})$. 

In this example, the number $\sqrt{d+1} = 2\sqrt{6}$ is in the Hilbert class field $L_{(1)}$, so we simplify our calculation somewhat by computing $g_{23}^\ast(x) = g_{23}\left(2\sqrt{6}x\right)$ instead of $g_{23}(x)$, and by computing a polynomial $h_{23}^\ast(x)$ for the action of a Galois group generator on a root of $g_{23}^\ast(x)$.
Values of $f_{23}(x)$, $g_{23}^\ast(x)$, and fiducial vectors $v_{\rm 23b}$ and $v_{\rm 23f}$ for 23b and 23f may be found in the accompanying text file \texttt{dim23.txt}; the value of $h_{23}^\ast(x)$ may be found in \texttt{h23star.txt}.


\bibliographystyle{plain}{}
\bibliography{references}


\end{document}